\documentclass[oneside,english]{amsart}
\usepackage[T1]{fontenc}
\usepackage[latin9]{inputenc}
\usepackage{amsthm}
\usepackage{amstext}
\usepackage{amssymb}
\usepackage{esint,graphicx}

\makeatletter
\numberwithin{equation}{section}
\numberwithin{figure}{section}
\theoremstyle{plain}
\newtheorem{thm}{\protect\theoremname}
  \theoremstyle{remark}
  \newtheorem{rem}[thm]{\protect\remarkname}
 \theoremstyle{definition}
 \newtheorem*{defn*}{\protect\definitionname}
  \theoremstyle{plain}
  \newtheorem{lem}[thm]{\protect\lemmaname}
  \theoremstyle{definition}
  \newtheorem{defn}[thm]{\protect\definitionname}
  \theoremstyle{plain}
  \newtheorem{prop}[thm]{\protect\propositionname}
  \theoremstyle{plain}
  \newtheorem{cor}[thm]{\protect\corollaryname}

\makeatother

\usepackage{babel}
  \providecommand{\corollaryname}{Corollary}
  \providecommand{\definitionname}{Definition}
  \providecommand{\lemmaname}{Lemma}
  \providecommand{\propositionname}{Proposition}
  \providecommand{\remarkname}{Remark}
\providecommand{\theoremname}{Theorem}

\begin{document}

\title[Conformal Invariance of Ising crossing probabilities with free b.c.]{Conformal invariance of crossing probabilities for the Ising model with free boundary conditions}

\author{St\'ephane Benoist, Hugo Duminil-Copin and Cl\'ement Hongler}

\begin{abstract}
We prove that crossing probabilities for the critical planar Ising model with free boundary conditions are conformally invariant in the scaling limit, a phenomenon first investigated numerically by Langlands, Lewis and Saint-Aubin \cite{LaLeSA_UnivIs}. We do so by establishing the convergence of certain exploration processes towards SLE$(3,\frac{-3}2,\frac{-3}2)$. We also construct an exploration tree for free boundary conditions, analogous to the one introduced by Sheffield \cite{Sh_ExplTreesCLE}.
\end{abstract}

\maketitle

\section{Introduction}

\subsection{Definition of the Ising model}

In this article, $\mathbb Z^2$ denotes the integer lattice $\{x=(x_1,x_2), x_1,x_2\in\mathbb Z\}$. Two vertices $x$ and $y$ of $\mathbb Z^2$ are neighbors if $\|x-y\|_1:=|x_1-y_1|+|x_2-y_2|=1$. In such case, we write $x\sim y$. Each subset $\mathcal{G}$ of $\mathbb Z^2$ can be seen as a graph by considering the graph induced by $\mathbb Z^2$ on the vertex set $\mathcal{G}$ : the edge-set $\mathcal{E}_\mathcal{G}$ of $\mathcal{G}$ consists of all edges of the lattice $\mathbb Z^2$ that links two vertices of $\mathcal{G}$ together. Let us also consider the dual graph $(\mathbb Z^2)^*$ of $\mathbb Z^2$ whose vertices sit at the center of faces of $\mathbb Z^2$, and whose edges are in one-to-one correspondence with edges in $\mathbb Z^2$. We define the boundary of $\mathcal{G}$ to be the set of edges $\partial \mathcal{G}=\{e=(x,y), x\sim y \text{ such that } x\in \mathcal{G}\text{ and } y\notin \mathcal{G}\}$. We sometimes abusively identify a boundary edge $(x,y)$ with one of its endpoint. 

For a domain $\Omega \subset \mathbb{C}$ - i.e. an open subset of the plane, we build its discrete approximation $\Omega^\delta$ of mesh size $\delta$ as the `big connected component' of the graph $\Omega\cap(\delta\mathbb Z^2)$ (for example $\Omega^\delta$ is the connected component containing a marked interior point $x_0\in\Omega$).
 Finally, define $(\Omega^\delta)^*$ to be the subgraph of $\delta(\mathbb Z^2)^*$ generated by vertices corresponding to faces of $\delta\mathbb Z^2$ that are either included in or adjacent to a face of $\Omega^\delta$.

The Ising model is one of the most classical models in equilibrium
statistical mechanics. A configuration of the Ising model in a domain $\Omega^\delta$ is an element of $\{-1,1\}^{\Omega^\delta}$ - sometimes denoted by $\{-,+\}^{\Omega^\delta}$, and the Ising model at inverse-temperature $\beta>0$ with free boundary condition is given by the probability measure

$$
\mu^{\rm free}_{\Omega^\delta,\beta}(\sigma)=\frac1{Z^{\rm free}(\Omega^\delta,\beta)}\exp\Big(\beta\sum_{\substack{x,y\in\Omega^\delta\\ x\sim y}}\sigma_x\sigma_y\Big),
$$
where the partition function $Z^{\rm free}(\Omega^\delta,\beta)$ is defined so that $\mu^{\rm free}_{\Omega^\delta,\beta}$ is a probability measure.
One may define an infinite-volume measure $\mu^{\rm free}_\beta$ on $\{-1,1\}^{\delta \mathbb{Z}^2}$ by taking the weak limit of finite-volume measures.

As was predicted
by Kramers and Wannier \cite{KW_ising1}, and shown by Onsager
\cite{Onsager_Cry}, a phase transition occurs at $\beta_{c}=\frac{1}{2}\ln\left(1+\sqrt{2}\right)$:
\begin{itemize}
\item For $\beta<\beta_{c}$, there exists $\tau=\tau(\beta)>0$ such that for any $\delta>0$ and any $x,y\in\Omega^\delta$,
$$\mu^{\rm free}_{\beta}(\sigma_{x}\sigma_{y})\le \exp(-\tau\|x-y\|_1/\delta).$$ 
\item For $\beta>\beta_{c}$, there exists $m=m(\beta)>0$ such that for any $\delta>0$ and any $x,y\in\Omega^\delta$,
$$\mu^{\rm free}_\beta(\sigma_{x}\sigma_{y})\ge m.$$
\end{itemize}
The fine description of what happens at and near the critical point has been the subject of more than sixty years of investigation. Two successful approaches enabled mathematicians and physicists to study this critical phase. On the one hand, the exact solution led to many explicit formulae for spin-spin correlations and other thermodynamical quantities; see \cite{Baxter_exact,MW_ising,Palmer_planar} and references therein for further details. On the other hand, the scaling limit of the model was conjectured to be conformally invariant using renormalization group arguments \cite{Fi_Renormalization,FrMaSe_CFT}, a prediction which led to a deep understanding of the critical phase, albeit non-rigorous. In recent years, conformal invariance of the Ising model became the object of an intense mathematical effort. Chelkak and Smirnov \cite{Smi_ICM,SmiChe_Ising} proved conformal invariance of the so-called fermionic observable, a property which led to the proof of convergence of interfaces in domains with Domain-Wall boundary conditions (or Dobrushin boundary conditions) to the so-called Schramm-Loewner evolution (SLE) \cite{ChDuHoKeSm_Ising}. The spin-spin correlations were also studied \cite{ChHoIz_SpinCorrIs}. This note belongs to this effort, and studies the probability of crossing events (see definition below).

\subsection{Main results}
Crossing events are macroscopic observables describing the connectivity properties
of a random configuration. Formally, let $(\Omega,a,b,c,d)$ be a topological rectangle, i.e. a simply connected Jordan domain $\Omega$ with four points on its boundary, indexed in clockwise order. Note that $a$, $b$, $c$ and $d$ determine four arcs on the boundary denoted by $[ab]$, $[bc]$, $[cd]$ and $[da]$.
Let $a^\delta$, $b^\delta$, $c^\delta$ and $d^\delta$ be the vertices of $\partial\Omega^\delta$ closest to $a$, $b$, $c$ and $d$ respectively.
The rectangle $(\Omega^\delta,a^\delta,b^\delta,c^\delta,d^\delta)$ is {\em crossed} in an Ising configuration $\sigma$ if there exists a path of plusses going from $[a^\delta b^\delta]$ to $[c^\delta d^\delta]$, i.e. if there exists a sequence of vertices $v_0,\dots,v_n\in\Omega^\delta$ such that 
\begin{itemize}
\item $v_i$ and $v_{i+1}$ are neighbors for any $0\le i<n$,
\item $v_0\in[a^\delta b^\delta]$ and $v_n\in[c^\delta d^\delta]$,
\item $\sigma_{v_i}=+$ for any $0\le i\le n$.\end{itemize}
We denote such a crossing event by $\left[a^\delta b^\delta\right]{\leftrightsquigarrow}\left[c^\delta d^\delta\right]$, and call its probability a {\em crossing probability}.

Let us define a slight variation of the previous event. Two vertices $x$ and $y$ of $\Omega^\delta$ are called $\star$-neighbors if $\|x-y\|_\infty:=\max\{|x_1-y_1|,|x_2-y_2|\}=\delta$. With this definition, each vertex has eight neighbors instead of four. The rectangle $(\Omega^\delta,a^\delta,b^\delta,c^\delta,d^\delta)$ is {\em $\star$-crossed} if there exists a $\star$-path of plusses going from $[a^\delta b^\delta]$ to $[c^\delta d^\delta]$.
The event is denoted by $\left[a^\delta b^\delta\right]\overset{\star}{\leftrightsquigarrow}\left[c^\delta d^\delta\right]$.
\medbreak
The first theorem of this paper yields that crossing probabilities for the critical Ising model with free boundary conditions converge, when the mesh size tends to 0, to a conformally invariant limit.

\begin{thm}\label{thm:main}
There exists a function $f$ from the set of topological rectangles to $[0,1]$ such that
\begin{itemize}
\item for any topological rectangle $(\Omega,a,b,c,d)$,

\begin{eqnarray*}
\lim_{\delta\to0}\mu_{\Omega^{\delta},\beta_c}^{\mathrm{free}}\big( \left[a^\delta b^\delta\right]{\leftrightsquigarrow}\left[c^\delta d^\delta\right]\big) =\lim_{\delta\to0}\mu_{\Omega^{\delta},\beta_c}^{\mathrm{free}}\big( \left[a^\delta b^\delta\right]\overset{\star}{\leftrightsquigarrow}\left[c^\delta d^\delta\right]\big)=f(\Omega,a,b,c,d)
\end{eqnarray*}

\item $f$ only depends on the conformal type, i.e. for any topological rectangle $(\Omega,a,b,c,d)$ and any conformal map $\Phi:\Omega\rightarrow\mathbb C$, 

$$
f(\Phi(\Omega),\Phi(a),\Phi(b),\Phi(c),\Phi(d))=f(\Omega,a,b,c,d)
$$

\end{itemize}
\end{thm}

 The first property guarantees that connectivity and $\star$-connectivity are the same in the scaling limit. Let us also mention that the result should extend to isoradial graphs, thus proving some sort of universality (see \cite{SmiChe_isoradial} for a precise definition of isoradial graphs).

Crossing probabilities played a central role in the study of planar lattice models at criticality. The work of Langlands, Pouliot and Saint-Aubin \cite{LPS},
who verified numerically Cardy's formula \cite{Ca} for percolation crossing
probabilities, constitutes one of the first direct evidences of
full conformal invariance of lattice models. This formula
was then related non-rigorously by Schramm to SLE \cite{Law} and proved rigorously by Smirnov \cite{Smi_ICM} for critical site percolation on the triangular lattice. Let us stress out that this result was the crucial step towards the proof of conformal invariance of percolation interfaces.

In the case of the Ising model, crossing probabilities with free boundary conditions
were also investigated numerically by Langlands, Lewis and Saint-Aubin \cite{LaLeSA_UnivIs}.
They concluded to the conformal invariance and the universality of these probabilities. 

Unlike in the percolation case, no prediction for the limiting value of crossing probabilities is currently available for the Ising model with free boundary conditions. Indeed, the simplest generalization of Cardy's formula in Conformal Field Theory deals with crossing probabilities in topological rectangles $(\Omega^\delta,a^\delta,b^\delta,c^\delta,d^\delta)$ with spins fixed to be $+$ on $[a^\delta b^\delta]$ and $[c^\delta d^\delta]$, and $-$ on $[b^\delta c^\delta]$ and $[d^\delta a^\delta]$, and was studied in \cite{Iz_PhD}. 

While Theorem~\ref{thm:main} does not give an explicit formula, the proof provides some information on the limiting probabilities.
Indeed, the crossing probabilities for the Ising model with free
boundary conditions can be represented as hitting probabilities for
a discrete exploration process, which in words is the leftmost interface between $+$ and $-$, bouncing off the boundary in such a way that it can always reach $[b^\delta c^\delta]$ without crossing itself. Let us define it formally.

Consider a configuration $\sigma\in\{-1,1\}^{\Omega^\delta}$, and two vertices $u$ and $v$ of $\partial(\Omega^\delta)^*$. An exploration process $\gamma^\delta:\{0,\dots,n\}\rightarrow(\Omega^\delta)^*$ from $u$ to $v$ is a path such that :
\begin{itemize} 
\item $\gamma^\delta_0=u$, $\gamma^\delta_n=v$, and $\gamma^\delta_i\sim\gamma^\delta_{i+1}$ for any $0\le i \le n$.
\item $\gamma^\delta$ is non self-crossing. In particular, the vertex $\gamma^\delta_{j+1}$ is always in the connected component of $(\Omega^\delta)^*\setminus\gamma^\delta[0,j]$ containing $v$.
\item The vertex to the left of the oriented edge $(\gamma^\delta_j,\gamma^\delta_{j+1})$ is either outside of $\Omega^\delta$ or carries a $+$ spin.
\item The vertex to the right of the oriented edge $(\gamma^\delta_j,\gamma^\delta_{j+1})$ is either outside of $\Omega^\delta$ or carries a $-$ spin.
\end{itemize}

\begin{rem}
Note that we could have chosen our exploration paths to let $+$ spins on their right instead. Such an exploration path from $a$ to $b$ would give, by time-reversal, an exploration path from $b$ to $a$ letting $+$ spins on its left.
\end{rem}

Let us also describe a specific choice of exploration that we will consider.

\begin{defn}We denote by $\gamma^{\ell,\delta}=\gamma^{\ell,\delta}_{u,v}$ the {\em leftmost free explorer} from $u$ to $v$, i.e. the leftmost of all exploration paths with these endpoints. Such a path satisfies the property that, given $\gamma^{\ell,\delta}[0,j]$, the following step $\gamma^{\ell,\delta}_{j+1}$ is always the counterclockwisemost admissible step for an exploration process.
\end{defn}

By construction, the leftmost explorer is a non self-crossing curve drawn on the dual lattice starting from $u$ and ending at $v$. On the boundary, it turns in such a way that it can always end at $v$ without ever crossing itself.  Inside the domain, it is an interface between a path of $+$ spins and a $\star$-path of $-$ spins, or equivalently it is an exploration that turns left whenever there is an ambiguity.

Similarly, one may construct the {\em rightmost free explorer} $\gamma^{r,\delta}_{u,v}$ by letting our explorer always make the clockwisemost admissible turn. Note that any exploration process is sandwiched between $\gamma^{\ell,\delta}$ and $\gamma^{r,\delta}$.

\medbreak
Define the {\em Continuous Dipolar Explorer} (CDE) in $\mathbb H$ from 0 to $\infty$ to be the SLE$(3,-\tfrac32,-\tfrac32)$ process with driving points $0^-$ and $0^+$ (see next section for the definition of SLE$(\kappa,\rho_1,\rho_2)$). The CDE in a domain $\Omega$ from a boundary point $u$ to another boundary point $v$ is then the image of the CDE in $\mathbb H$ from 0 to $\infty$ under any conformal bijection from $(\mathbb{H},0,\infty)$ to $(\Omega,u,v)$.

\begin{thm}\label{thm:dde-to-cde-cv}
Let $\Omega$ be a simply connected Jordan domain, with two marked points $u$ and $v$ on its boundary. Consider the critical Ising model on $\Omega^\delta$ with free boundary conditions. Any family of exploration processes $(\gamma^{\delta}_{u^\delta,v^\delta})$ converges in law (as the mesh size $\delta$ goes to $0$) to the {\rm CDE} from $u$ to $v$ in $\Omega$. In particular, $(\gamma^{\ell,\delta}_{u^\delta,v^\delta})$ and $(\gamma^{r,\delta}_{u^\delta,v^\delta})$ converge to the same limit.
\end{thm}

In the previous theorem, the topology on curves used in the definition of the convergence in law is the supremum norm up to time reparametrization. Specifically, for a simply-connected Jordan domain $\Omega$ with two marked boundary points $u$ and $v$, we work with the set $\mathcal{C}(\overline{\Omega})$ of continuous curves $\gamma : [0,1] \to \overline\Omega$ considered up to time reparametrization. The distance between two such curves $\gamma_1$ and $\gamma_2$ is then defined as
$$
d(\gamma_1,\gamma_2) := \inf_{\varphi} \sup_{t\in [0,1]}|\gamma_1(t) - \gamma_2(\varphi_t)|,
$$
where the infimum runs over all increasing (bicontinuous) bijections $\varphi$ of $[0,1]$ onto itself. A similar definition makes sense if curves are parametrized by $\mathbb{R}^+\cup\{\infty\}$ instead of $[0,1]$.

The proof of Theorem~\ref{thm:dde-to-cde-cv} relies on recent results on Ising model \cite{ChDuHo_CrosProbIs,HoKy_Isintfreebc} and on ideas found in \cite{Sh_ExplTreesCLE,MilSheIG1}. 

We now state a corollary of Theorem~\ref{thm:dde-to-cde-cv} that itself implies Theorem~\ref{thm:main}.

\begin{cor}\label{cor:1}
For any topological rectangle $(\Omega,a,b,c,d)$, the crossing probability is given by
\begin{equation*}f(\Omega,a,b,c,d):=\mathbb P\Big({\rm CDE}\mbox{ hits }\left[cd\right]\mbox{ before }\left[bc\right]\Big),\end{equation*}
where the \emph{CDE} goes from the point $a$ to an arbitrary point $v\in[bc]$.
\end{cor}

\subsection{Exploration tree}

Discrete explorers allow one to build an exploration 'tree' analogous to the process defined by Sheffield \cite{Sh_ExplTreesCLE}. This object describes the exploration 'arcs' touching the boundary
of a domain, and is defined as follows. For any couple of boundary points $u,v \in \partial(\Omega^\delta)^*$, consider the set $\Gamma_{u,v}\subset \mathcal{C}\left(\overline{\Omega}\right)$ of all exploration paths (leaving $+$ spins on their left) from $u$ to $v$. Define the discrete free arc ensemble by 

$$
\mathcal{A}^\delta := \bigcup_{u,v\in \partial(\Omega^\delta)} \{u\}\times\{v\}\times\Gamma_{u,v} \subset \overline{\Omega}\times\overline{\Omega}\times\mathcal{C}\left(\overline{\Omega}\right)
$$

The last theorem of this article deals with the convergence of this discrete free arc ensemble towards the Free Arc Ensemble (FAE) the definition thereof we postpone to Section \ref{sec:4}.

\begin{thm}\label{thm:free-arc-cv}
Let $\Omega$ be a simply connected Jordan domain. Consider the critical Ising model on $\Omega^\delta$ with free boundary conditions.
The family $(\mathcal{A}^{\delta})_{\delta>0}$ converges in law to the {\rm FAE}, as the mesh size $\delta$ goes to $0$.
\end{thm}

The topology used in the convergence in law result will be given in Section \ref{sec:4}.

\begin{rem}The convergence in Theorem \ref{thm:free-arc-cv} is robust: the arguments we give hold for any family of discrete domains $\Omega^\delta$ that converges towards $\Omega$ in the sense of convergence of their boundaries (as curves, for the topology of uniform convergence up to reparametrization).
\end{rem}

\subsection{Organization of the paper}

The paper is organized as follows. In Section~\ref{sec:1}, we remind the definition of the Continuous Dipolar Explorer. We state two useful results on the Ising model in Section \ref{sec:bg}. In Section~\ref{sec:2}, we prove Theorem~\ref{thm:dde-to-cde-cv}. Section~\ref{sec:3} is devoted to the proof of Theorem~\ref{thm:main}. In Section~\ref{sec:4} we define the {\em Free Arc Ensemble} and we discuss how to deduce Theorem~\ref{thm:free-arc-cv} from Theorem~\ref{thm:dde-to-cde-cv}.

\section{Definition of Schramm-Loewner evolutions driven by several points}\label{sec:1}

\subsection{Loewner chains}

Loewner chains allow one to encode a growing compact set in the upper half-plane $\mathbb H:=\{z\in\mathbb{C}:\Im{\rm m}(z)>0\}$ by one real-valued function. In particular, one can encode in this way simple curves between two boundary points of $\mathbb H$. We refer to \cite{Law} for a book on this subject.

Let $\left(\gamma_s\right)_{s\geq0}$ be a continuous curve in $\overline{\mathbb{H}}$
such that $\gamma_0=0$ and $\gamma_s\rightarrow \infty$ as $s\rightarrow\infty$. Let $H_{s}$ be the unbounded connected component
of $\mathbb{H}\setminus\gamma\left[0,s\right]$. Consider the conformal bijection $g_{s}:H_{s}\to\mathbb{H}$ normalized in such a way that 
$$g_{s}\left(z\right)=z+\frac{2a_{s}}{z}+o\left(\frac{1}{z}\right)$$ as $z\to\infty$. In the situations we will consider, the half-plane capacity ($h$-capacity) $a_s$ is a continuous increasing bijection of $\mathbb{R}^+$, and we can then reparametrize by $t = a_s$. For simplicity, we will reserve the notation $t$ for parametrizations such that $t$ is the $h$-capacity of the hull at time $t$.

The sequence of functions $\left(g_{t}\right)_{t\geq0}$ is then solution to the
Loewner flow equation
\begin{eqnarray*}
\partial_{t}g_{t}\left(z\right) & = & \frac{2}{g_{t}\left(z\right)-U_{t}}\\
g_{0}\left(z\right) & = & z,
\end{eqnarray*}
where $U_{t}=g_{t}\left(\gamma_{t}\right)\in\mathbb{R}$ is the so-called
{\em driving function} of the Loewner chain.

Conversely, any real-valued function $\left(U_{t}\right)_{t\geq0}$
defines, when solving the Loewner flow equation above, a Loewner chain $\left(g_{t}:H_{t}\to\mathbb{H}\right)_{t\geq0}$, where $H_t$ can be recovered as the set of initial conditions $z\in \mathbb H$ for which the differential equation does not blow up before time $t$.
If $\left(U_{t}\right)_{t\geq0}$ is regular enough (see \cite[Chapter 4.4]{Law} for a precise statement), there exists
a curve $\left(\gamma_t\right)_{t\geq0}$ such that $H_{t}$ is
the unbounded component of $\mathbb{H}\setminus\gamma\left[0,t\right]$ (and we call $K_t = \mathbb{H}\setminus H_t$ the hull generated by $\gamma[0,t]$). In such case, the Loewner chain $\left(g_{t}\right)_{t\geq0}$ is
said to be {\em generated by the curve} $\left(\gamma_t\right)_{t\geq0}$, and $(U_t)_{t\ge0}$ gives rise to a parametrized curve in $\overline{\mathbb H}$ from 0 to $\infty$.

\subsection{Loewner chains driven by a stochastic process} 

In \cite{S0}, Schramm suggested to build random curves by looking at Loewner chains generated by certain random driving functions. These random growing compact sets are called Schramm-Loewner evolutions. We now discuss three important examples. In the following the process $(B_t)$ is always a standard one-dimensional Brownian motion.

\subsubsection{SLE$(\kappa)$} Let $\kappa>0$. The chordal SLE$\left(\kappa\right)$ in $\mathbb H$ from 0 to $\infty$ is the Loewner chain driven by $U_{t}:=\sqrt{\kappa}B_{t}$.

\subsubsection{SLE$(\kappa,\rho)$}\label{subsec:slekr} Let $\kappa>0$ and $\rho\in\mathbb R$. The SLE$(\kappa,\rho)$ in $\mathbb H$ starting from 0 with {\em force point} $x\ge 0$ and {\em observation point} $\infty$ is the Loewner chain driven by $(U_t)_{t\ge0}$, where
$\left(U_{t},O_{t}\right)_{t\geq0}$ is the solution to the stochastic
differential equation
\begin{eqnarray}\label{eq:sle}
\mathrm{d}U_{t} & = & \sqrt{\kappa}\mathrm{d}B_{t}+\frac{\rho}{U_{t}-O_{t}}\mathrm{d}t,\\
\label{eq:sle2}\mathrm{d}O_{t} & = & \frac{2}{O_{t}-U_{t}}\mathrm{d}t,
\end{eqnarray}
with initial conditions $U_0=0$ and $O_0=x$.

The SLE$\left(\kappa,\rho\right)$ is a priori well-defined
by the above equation
until the first time when $U_{t}=O_{t}$.
It is however sometimes possible, depending on the values of the parameters $(\kappa,\rho)$, to find a reasonable solution of this system of SDEs defined for all times. Let us provide additional details.

Suppose $(U_t,O_t)_{t\ge0}$ satisfies the system of SDEs \eqref{eq:sle}/\eqref{eq:sle2} and let $X^\kappa_t=O_t-U_t$. Note that the rescaled process $X_t=X^\kappa_t/\sqrt{\kappa}$ satisfies the Bessel stochastic differential equation of dimension $d=1+\frac{2(\rho+2)}{\kappa}$ :
\begin{equation}\label{eq:bessel}
\mathrm{d}X_{t} ~ = ~ -\mathrm{d}B_{t}+\frac{\rho+2}{\kappa X_t}\mathrm{d}t.
\end{equation}
If moreover $(U_t,O_t)_{t\ge0}$ comes from a Loewner chain, the geometry forces $X_t$ to be non-negative, and instantaneously reflected at $0$ (i.e. the set of times at which $X_t=0$ is of Lebesgue measure $0$). One can check that -- provided that the dimension $d$ is strictly positive -- there is a unique\footnote{Uniqueness in law of $(X_t)$ is enough for what we need, but pathwise uniqueness would hold as well.} non-negative process $(X_t)_{t\ge0}$, called the Bessel process of dimension $d$, which is  instantaneously reflected at $0$ and which evolve according to \eqref{eq:bessel} whenever this equation is non-singular. Indeed, \eqref{eq:bessel} characterizes the process of excursions of $(X_t)_{t\ge0}$ out of $0$. The only degree of freedom we could have is when glueing together the excursions to recover $(X_t)_{t\ge0}$. But there is at most one way to glue a given ordered set of excursions to get an instantaneous reflected process.
Now that we are in possession of $(X_t)_{t\ge0}$, let us explain how we recover $(U_t,O_t)_{t\ge0}$. When $d>1$ (i.e. $\rho> -2$) the integral
$$\int_0^t\frac{1}{X_s}ds$$
is almost surely finite for any $t\geq 0$. Let us define the process $O_t=\tfrac1{\sqrt\kappa}\int_0^t\frac{2}{X_s}ds$. The process $X_t$ can be seen to solve the Bessel equation \eqref{eq:bessel} in its integral form
\begin{equation}\label{eq:bessel-int}
X_{t} ~ = ~ -B_{t}+\int_0^t\frac{\rho+2}{\sqrt{\kappa} X_s}\mathrm{d}s~=~-B_t+(\rho+2)O_t.
\end{equation}
If we let $U_t = O_t - \sqrt\kappa X_t$, we get the unique solution $(U_t,O_t)_{t\ge 0}$ of the system of SDEs \eqref{eq:sle}/\eqref{eq:sle2} such that $O-U$ is non-negative and instantaneously reflected at $0$. Hence, when $\rho>-2$, we have a somehow unique notion of an SLE($\kappa,\rho$) defined for all time.

\begin{rem}\label{rem:regularity-bessel}
The Bessel process $X_t$ of dimension $d>1$, as Brownian motion, is $\alpha$-H\"older for any $\alpha<\frac{1}{2}$. The integral process $O_t$ inherits the same H\"older regularity by equation \eqref{eq:bessel-int}. Moreover, the Hausdorff dimension of the set of zeroes of $X_t$ is $\frac{2-d}{2}=\frac{1}{2}-\frac{\rho+2}{\kappa}$.
\end{rem}

\subsubsection{SLE($\kappa,\rho_1,\rho_2$)} Let $\kappa>0$ and $\rho_1,\rho_2\in\mathbb R$. The SLE$(\kappa,\rho_1,\rho_2)$ in $\mathbb H$ starting from 0 with force points $\ell\le 0$ and $r\ge0$ and observation point $\infty$ is the Loewner chain driven by $(U_t)_{t\ge0}$, where $(U_t,O_t^L,O_t^R)_{t\ge0}$
is the solution of the system of SDEs
\begin{eqnarray}
\mathrm{d}U_{t} & = & \sqrt{\kappa}\mathrm{d}B_{t}+\left(\frac{\rho_1}{U_{t}-O_{t}^{L}}+\frac{\rho_2}{U_{t}-O_{t}^{R}}\right)\mathrm{d}t,\label{eq:sle-krr-du}\\
\mathrm{d}O_{t}^{L} & = & \frac{2}{O_{t}^{L}-U_{t}}\mathrm{d}t,\label{eq:sle-krr-dol}\\
\mathrm{d}O_{t}^{R} & = & \frac{2}{O_{t}^{R}-U_{t}}\mathrm{d}t,\label{eq:sle-krr-dor}
\end{eqnarray}
and initial conditions $U_0=0$, $O_0^L=\ell$ and $O_0^R=r$. As before, the solution is a priori defined only for times when $O_{t}^{L}<U_{t}<O_{t}^{R}$. Nevertheless, if $\rho_1,\rho_2>-2$, Miller and Sheffield \cite[Section 2.2]{MilSheIG1} showed that there is a unique reasonable solution to this system of SDEs, solution which is defined for all time.

More precisely, there is a unique (in law) triplet of processes satisfying the following properties:
\begin{itemize}
\item[P1] $(U_t,O_t^L,O_t^R)_{t\ge0}$ satisfies the three equations above on the set of times $t$ for which $O^L_t<U_t<O^R_t$, $O^L_t<U_t$ and $U_t<O^R_t$ respectively.
\item[P2] $O^L_t\le U_t\le O^R_t$ for any $t\ge 0$.
\item[P3] The process $U_t$ is instantaneously reflected off the force points, in the sense that the set of times $t$ for which $U_t=O^L_t$ or $U_t=O^R_t$ is of zero Lebesgue measure.
\item[P4] The two equations \eqref{eq:sle-krr-dol} and \eqref{eq:sle-krr-dor} hold in their integral forms, namely

$$
O_t^L=\int_0^t\frac{2}{O_{s}^{L}-U_{s}}\mathrm{d}s  \ \ \ \ \ \ \ O_t^R=\int_0^t\frac{2}{O_{s}^{R}-U_{s}}\mathrm{d}s  \ \ \ .
$$
\end{itemize}

\begin{lem}Under the assumption that the three properties {\rm P1--3} hold, the last property {\rm P4} is equivalent to
\begin{itemize}
\item[{\rm P4'}] The process $(U_t)_{t\ge0}$ is $\alpha$-H\"older continuous for every $\alpha<1/2$.
\end{itemize}
\end{lem}

\begin{proof}On the one hand, the property P4' holds for the unique solution of P1-4 as explained in Remark~\ref{rem:regularity-bessel}.

In order to prove that P4' implies P4, we can proceed as follow. Let us focus on the equation in P4. Issues may only arise at positive times when $O^L=U$. At these times, we necessarily have $O^R\neq U$. Hence, by the Girsanov theorem, the term $\frac{\rho_2}{U_{t}-O_{t}^{R}}$ of \eqref{eq:sle-krr-du} can be absorbed in the Brownian motion via a change of measure, and we can temporarily forget the point $O^R$. We are thus reduced to the setup of SLE$(\kappa,\rho)$ as in Section \ref{subsec:slekr}. The process $O^L-U$ under this change of measure is a Bessel process (up to multiplication by a constant).
Let $I_t =O_t^L-\int_0^t\frac{2}{O_{s}^{L}-U_{s}}\mathrm{d}s$.
Thanks to properties of the Bessel process (Remark \ref{rem:regularity-bessel}) and the property P4', we see that $I_t$ is an $\alpha$-H\"older process for any $\alpha<1/2$ that is constant outside of a set of Hausdorff dimension $\frac{1}{2}-\frac{\rho+2}{\kappa}<\frac{1}{2}$. Hence $I_t$ identically vanishes.
\end{proof}

\begin{rem}\label{rem:SLEcoord}
Let $(\Omega,s,l,r,o)$ be a simply-connected domain with four marked points on its boundary. The SLE$(\kappa,\rho)$ from $s$ with force point $l$ and observation point $r$ and the SLE$(\kappa,\rho, \kappa-6-\rho)$ from $s$ with force points $l$ and $r$ and with observation point $o$ have the same law until the first disconnection time of $r$ and $o$ (see e.g. \cite{SchWi_SLECoordChange}). In particular, the processes SLE$(3,\frac{-3}2)$ and SLE$(3,\frac{-3}2,\frac{-3}2)$ are related.
\end{rem}

\subsection{Definition of the CDE}

\begin{defn}
\label{def:two-point-cde}The CDE in $(\mathbb H,0,\infty)$ is the process SLE$(3,\frac{-3}2,\frac{-3}2)$ from 0 with force points $\ell=0^-$ and $r=0^+$ and observation point $\infty$ .  Let $\left(\Omega,a,b\right)$ be a simply
connected domain with two marked points $a,b\in\partial\Omega$. The CDE in $\left(\Omega,a,b\right)$ is the image under a conformal map from $\mathbb H$ onto $\Omega$ mapping 0 to $a$ and $\infty$ to $b$.\end{defn}

Let us recall a simple property of the CDE, namely that it is independent of the observation point.

\begin{prop}\label{prop:dipolar}
Consider a domain $\Omega$ with three marked boundary points $u$, $v_1$ and $v_2$. There exists a coupling of a {\rm CDE} $\gamma_1$ in $\left(\Omega,u,v_1\right)$ with a {\rm CDE} $\gamma_2$ in $\left(\Omega,u, v_2\right)$
such that the two curves coincide up to the first time when $v_1$ and $v_2$ get disconnected by the trace of the curve. \end{prop}

\section{Two results on the Ising model}\label{sec:bg}

In this section, we state two recent results about the Ising which used in an essential manner in this paper. 

\subsection{A result on convergence of interfaces}
The first result states the convergence of certain interfaces in the Ising model.

Let $(\Omega,b,\ell,r)$ be a simply-connected Jordan domain with three marked points on its boundary (indexed in clockwise order).
We consider the critical Ising model on discrete approximations of this domain, with boundary conditions $+$ on the clockwise arc $[b^\delta \ell^\delta]$, free on $[\ell^\delta r^\delta]$ and $-$ on $[r^\delta b^\delta]$. Call $\gamma^\delta$ the leftmost explorer from $b^\delta$ to $\ell^\delta$.

\begin{thm}\cite[Theorem 1]{HoKy_Isintfreebc}\label{thm:ci}
The law of the leftmost explorer $\gamma^\delta$ until the first hitting time of $[\ell^\delta r^\delta]$ converges to the law of an {\rm SLE}$(3,\frac{-3}2,\frac{-3}2)$ from $b$ to $\ell$ until its first hitting time of $[\ell r]$. The convergence is moreover uniform in the domain (see the original paper).
\end{thm}

An alternative proof of this result was provided in \cite{Izyurov_free}.
\subsection{A result on crossing probabilities}
Let us recall a classical definition: for a topological rectangle $(\mathcal Q,a,b,c,d)$, the \emph{extremal length}
$\ell_{\mathcal{Q}}\left(\left[ab\right],\left[cd\right]\right)$
is defined as 
\[
\sup_{\rho\: :\: \Omega\to[0,\infty)}\frac{\inf_{p}\left(\int_{p}\rho\left|\mathrm{d}z\right|\right)^{2}}{\int_{\mathcal{Q}}\rho^{2}\mathrm{d}x\mathrm{d}y},
\]
where the infimum is taken over rectifiable paths $p$ from $\left[ab\right]$
to $\left[cd\right]$.

We call a {\em discrete topological rectangle} a subgraph of $\delta\mathbb Z^2$ whose complement is connected, with four marked vertices on its boundary.

\begin{thm}\cite[Corollary 1.7]{ChDuHo_CrosProbIs}\label{thm:cp}
For each $M > 0$ there exists  $\eta=\eta (M) > 0$ such that the following holds: for any discrete topological rectangle $(\mathcal{Q},a,b,c,d)$ with $\ell_{\mathcal{Q}}\left(\left[ab\right],\left[cd\right]\right)\leq M$,
$$
\mu_{\mathcal{Q},\beta_c}^{\mathrm{mixed}}\Big([ab]{\leftrightsquigarrow}[cd]\Big) \geq \eta,
$$
where mixed boundary conditions mean free on $[ab]$ and $[cd]$, and $-$ on $[bc]$ and $[da]$.
\end{thm}

\section{Proof of Theorem~\ref{thm:dde-to-cde-cv} (Convergence of explorers)}\label{sec:2}

In this section, we fix a simply connected Jordan domain $(\Omega,u,v)$ and a conformal map $\Phi$ from $(\Omega,u,v)$ onto $(\mathbb H,0,\infty)$. We will first prove the convergence result of Theorem~\ref{thm:dde-to-cde-cv} for the leftmost explorer going from $u$ to $v$.

\medbreak
For $n\ge 0$, consider the {\em slit domain} $\Gamma_n$, i.e the subgraph of $\Omega^\delta$ constructed by removing all edges of $\Omega^\delta$ intersecting in their middle the dual-edges of $\gamma^\delta[0,n]$, and then taking the connected component of the new graph containing $v^\delta$.

The boundary of $\Gamma_n$ is composed of three arcs :\begin{itemize}
\item the arc $\mathcal C^+_n$ composed of vertices of $\partial\Gamma_n$ bordering $\gamma^\delta[0,n]$ on its left, hence carrying $+$ spins,
\item the arc $\mathcal C^-_n$ composed of vertices of $\partial\Gamma_n$ bordering $\gamma^\delta[0,n]$ on its right, hence carrying $-$ spins,
\item the arc $\mathcal C^{\rm free}_n=\partial\Gamma_n\setminus(\mathcal C^+_n\cup\mathcal C^-_n)$. Note that $\mathcal C^{\rm free}\subset \partial\Omega^\delta$.
\end{itemize}

Let $L_n^\delta$ (resp. $R_n^\delta$) be the leftmost (resp. rightmost) point of $\partial\Omega^\delta$ reached by $\gamma^\delta[0,n]$.
Note that  $\mathcal{C}^+_n=[\gamma^\delta_n L_n^\delta]$, $\mathcal{C}^-_n=[R_n^\delta\gamma^\delta_n]$ and $\mathcal{C}^{\mathrm{free}}_n=[L_n^\delta R_n^\delta]$.

The proof goes in two steps.

\begin{enumerate}
\item We first prove that the discrete explorations $(\gamma^{\ell,\delta},L^\delta,R^\delta)$ form a tight family of random variables.

\item Then, if $(\gamma^{\ell},L,R)$ is any sub-sequential limit, we can consider the growing hull in $\mathbb H$ generated by $\Phi(\gamma^{\ell})$, and let $(g_t)_{t\ge0}$ be the associated Loewner chain parametrized by $h$-capacity. We call $(U_t)_{t\ge 0}$ its driving process.

Let $O_t^L=g_t(\Phi(L_t))$ and $O_t^R=g_t(\Phi(R_t))$. We show that $(U,O^L,O^R)$ satisfies the conditions P1-4', thus identifying $\gamma^{\ell}$ as being CDE.
\end{enumerate}

This will prove the statement for the leftmost free explorer : since any sub-sequential limit is a CDE, the family of curves $(\gamma^{\ell,\delta})_{\delta>0}$ is convergent. 

\subsection{Step 1: tightness of the law of the leftmost explorer}

\begin{prop}\label{thm:Loewner}
The family $(\gamma^{\ell,\delta})$ is tight. Moreover, any sub-sequential limit $\gamma^{\ell}$ satisfies the following properties:

\begin{itemize}
\item $\gamma^\ell$ is almost surely a non-self-crossing curve.

\item $\gamma^\ell$ can be almost surely parametrized by the $h$-capacity of the hull $\hat K_s$ of $\Phi(\gamma^{\ell}[0,s])$.

\item Consider $(K_t)_{t\ge0}$  the hull $\hat K$ parametrized by its capacity. It is a Loewner chain with a driving process $(U_t)_{t\ge0}$ which is almost surely $\alpha$-H\"older continuous for any $\alpha<1/2$.
\end{itemize}
\end{prop}

\begin{proof} Let us first fix some notation.
Consider the family $\mathfrak Q_t$ of discrete topological rectangles $\mathcal{Q}\subset\Omega^\delta\setminus\gamma^{\ell,\delta}[0,t]$ such that the boundary arcs $\left[bc\right]$
and $\left[da\right]$ are on $\partial(\Omega^\delta\setminus\gamma^{\ell,\delta}[0,t])$. Call $\mathcal{Q}$
\emph{avoidable} if it does not disconnect $\gamma^{\ell,\delta}_t$ from $v$.  A sub-path $\gamma^{\ell,\delta}[t_0,t_1]$ {\em crosses} $\mathcal Q$ if there exist $t_0\le s_0\le s_1\le t_1$ such that $\gamma^{\ell,\delta}_{s_0}\in[ab]$ and $\gamma^{\ell,\delta}_{s_1}\in [cd]$. 

The proposition is the conclusion of \cite[Theorem 1.3]{KeSm_LoewnerE}. Hence, using \cite[Proposition 2.2 and Corollary 2.3]{KeSm_LoewnerE}, it suffices to check that the following condition holds:
\medbreak
\noindent \textbf{Condition C2.}
{\em For any $M>0$, there exists $\eta>0$ such that for any $\delta>0$,
any stopping time\footnote{A stopping time $\tau$ is a stopping time for the filtration generated by the curve $(\gamma^{\ell,\delta}_t)$ parametrized by the number of steps.
}  $\tau$ and any avoidable discrete topological rectangle $\mathcal{Q}\subset\mathfrak{Q}_{\tau}$
with $\ell_{\mathcal{Q}}\left(\left[ab\right],\left[cd\right]\right)\geq M$,
\[
\mu_{\Omega^{\delta},\beta_c}^{\mathrm{free}}\Big( (\gamma^{\ell,\delta}_t)_{t\ge \tau}\mbox{ crosses }\mathcal{Q}\Big|\gamma^{\ell,\delta}[0,\tau]\Big) \leq1-\eta.
\]
}

In order to prove that this condition is satisfied, fix $\delta>0$ and $\tau$, as well as a realization of $\gamma^{\ell,\delta}[0,\tau]$. We also fix an avoidable discrete topological rectangle $\mathcal{Q}$ with $\ell_{\mathcal{Q}}\left(\left[ab\right],\left[cd\right]\right)\geq M$.

The event that $\gamma^{\ell,\delta}$ does not cross $\mathcal{Q}$
contains the event that $\mathcal{Q}$ is crossed from $\left[bc\right]$
to $\left[da\right]$ by a crossing of $+$ spins or a crossing of
$-$ spins. Since the discrete rectangle is avoidable, it cannot intersect both boundary arcs with $+$ and $-$ boundary conditions (i.e. $\mathcal C^+_\tau$ and $\mathcal C^-_\tau$), since it would then disconnect $\gamma^{\ell,\delta}_\tau$ from $v$. Without loss of generality, assume that it intersects only the arc with $+$ and free boundary conditions (i.e. $\mathcal C^+_\tau$ and $\mathcal C^{\rm free}_\tau$). The spatial Markov property (or Gibbs property) of the Ising model together with monotonicity with respect to boundary conditions guarantee that
\begin{equation}\label{eq:ineqcross}
\mu_{\Omega^{\delta},\beta_c}^{\mathrm{free}}\Big( (\gamma^{\ell,\delta}_t)_{t\ge \tau}\mbox{ does not cross }\mathcal{Q}\ \Big|\ \gamma^{\ell,\delta}\left[0,\tau\right]\Big) \geq\mu_{\mathcal Q,\beta_c}^{\rm mixed}\Big( \left[bc\right]\leftrightsquigarrow\left[da\right]\Big),
\end{equation}
where the mixed boundary conditions are free on $[bc]$ and $[da]$, and $-$ on $[ab]$ and $[cd]$. Theorem \ref{thm:cp} shows that the right-hand side of \eqref{eq:ineqcross} is larger than some positive quantity $\eta(M)$.
\end{proof}

We will also need to keep track of the points $L^\delta$ and $R^\delta$ in the scaling limit. In order to do this, let us prove the following property of subsequential limits of $\gamma^{\ell,\delta}$, which states that the scaling limit cannot touch the boundary when the discrete explorer does not. We can assume that the curves $\gamma^{\ell,\delta}$ almost surely converge towards a limit $\gamma$ for the uniform topology (i.e. we assume that all the curves are parametrized in a compatible way).

\begin{lem}\label{lem:boundary_touch}
Any time when $\gamma^\ell$ is on the boundary is a limit of boundary hitting times of $\gamma^{\ell,\delta}$.
\end{lem}

\begin{proof}
The curve $\gamma^\ell$ can be parametrized by capacity (Proposition \ref{thm:Loewner}) hence it does not stay on the boundary. Any time when $\gamma^\ell$ is on the boundary is hence in the closure of the set of boundary hitting times, and we can reduce to proving the statement for such times.

Now, each time the explorer $\gamma^{\ell,\delta}$ is inside the domain and gets close to the boundary we can use crossing estimates (Theorem \ref{thm:cp}) to see that with high probability the explorer touches the boundary soon after (see Figure \ref{fig:bound}). We leave the details of the previous claim, which are classical, to the reader. As there are countably many boundary hitting times, this yields the claim.
\end{proof}

\begin{figure}[htb]
\begin{center}
\includegraphics[width = 10cm]{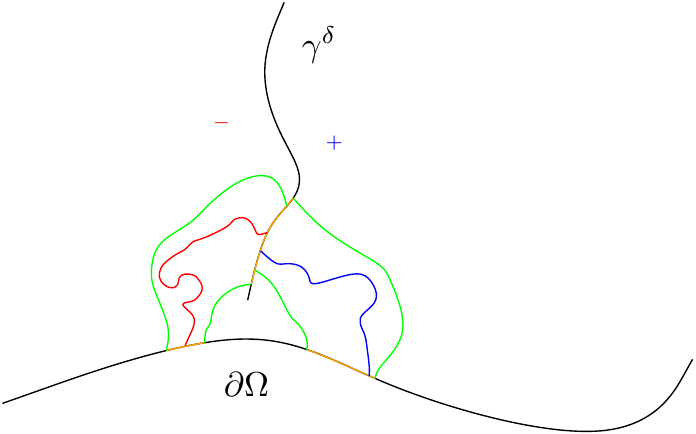}
\caption{With high probability, the rectangle to the left (resp. to the right) of $\gamma^\delta$ contain a crossing of $+$ (resp $-$) between its orange boundaries. These crossings forces the explorer $\gamma$ to touch the boundary $\partial \Omega$ in a neighborhood.}
\label{fig:bound}
\end{center}
\end{figure}

The previous results imply the following, where the topology used is the weak topology associated to uniform convergence (up to common reparametrization) of the triplet of functions.
\begin{cor}
On any subsequence such that $\gamma^{\ell,\delta}$ converges, the triplet $(\gamma^{\ell,\delta},L^\delta,R^\delta)$ converges to $(\gamma^\ell,L,R)$ where $L$ (resp. $R$) is the leftmost (resp. rightmost) boundary point reached by $\gamma^\ell$.
\end{cor}

\subsection{Step 2: identification of the scaling limit of the leftmost explorer}

Let $\gamma^\ell$ be a possible sub-sequential limit of $(\gamma^{\ell,\delta})_{\delta>0}$. Recall that we defined $O^L$ (resp. $O^R$) to be $g_t(\Phi(L_t))$ (resp. $g_t(\Phi(R_t))$) where $\Phi$ uniformizes $(\Omega,u,v)$ to the upper half-plane $(\mathbb{H},0,\infty)$, and where $g_t$ is the Loewner chain associated with $\Phi(\gamma^\ell)$. We also denote by $U_t$ the driving process of the Loewner chain $\Phi(\gamma^\ell)$.
We now want to prove that the triplet $(U,O^L,O^R)$ satisfies the four properties P1, P2, P3 and P4' that characterize SLE$(3,\frac{-3}2,\frac{-3}2)$.

\begin{proof}[Proof of Property P1]
The curve $\gamma^\ell$ is the limit of discrete interfaces $(\gamma^{\ell,\delta})_{\delta> 0}$ which satisfy a spatial Markov property. At any stopping time $\tau$ when $\gamma^{\ell,\delta}$ is away from $L^\delta$ and $R^\delta$, we are in the setup of Theorem \ref{thm:ci} and hence $(\gamma^{\ell,\delta})_{t\ge \tau}$ is close to an SLE$(3,\frac{-3}2,\frac{-3}2)$ in the slit domain $\Gamma_\tau=\Omega^\delta\setminus\gamma^{\ell,\delta}[0,\tau]$. In other words, if $t$ is a time such that $O^L_t<U_t<O^R_t$, the process $(g_t(\gamma^\ell_{t+s}))_{s\ge 0}$ up to the first hitting time of $(-\infty,O^L_t]\cup[O^R_t,\infty)$ is a SLE$(3,\frac{-3}2,\frac{-3}2)$ with force points $O^L_t$ and $O^R_t$. Hence, the equation \eqref{eq:sle-krr-du} is satisfied on the set of times $t$ such that $O^L_t<U_t<O^R_t$. Whenever $O^L_t<U_t$, the point $L$ is locally constant, hence $O^L_t=g_t(\Phi(L_t))$ simply evolves according to the Loewner flow equation \eqref{eq:sle-krr-dol}. Same goes for $O^R$ and \eqref{eq:sle-krr-dor}.
\end{proof}

\begin{proof}[Proof of Property P2]
The second property follows from the geometry : $O^L_t$ is the leftmost point absorbed by the Loewner chain. In particular, it always sits to the left of $U_t$. Same goes for $O^R_t$.
\end{proof}

\begin{proof}[Proof of Property P3]
We want the sets $\{t\ge 0:U_t=O^L_t\}$ and $\{t\ge 0:U_t=O^R_t\}$ to be (almost surely) of Lebesgue measure 0.
This is a deterministic property satisfied by any Loewner chain generated by a continuous driving function (we refer to \cite[Lemma 2.5]{MilSheIG1}).
\end{proof}

\begin{proof}[Proof of Property P4']
The last property corresponds to the third property of Proposition \ref{thm:Loewner}.
\end{proof}

This concludes the proof of the convergence of the leftmost explorer towards CDE. The same argument works as well for the rightmost explorer. The couple of discrete curves $(\gamma^{\ell,\delta}_{u^\delta,v^\delta},\gamma^{r,\delta}_{u^\delta,v^\delta})$ hence jointly converge to the same limit $\gamma$. Indeed, on the one hand $\gamma^r$ is always to the right of $\gamma^\ell$, but on the other hand, these two curves have the same law.

Let us now focus on families of arbitrary exploration processes. Any exploration process $\gamma^\delta$ from $u^\delta$ to $v^\delta$ is squeezed between $\gamma^{\ell,\delta}_{u^\delta,v^\delta}$ and $\gamma^{r,\delta}_{u^\delta,v^\delta}$ and thus converges towards a curve $\tilde{\gamma}$ which is almost surely equals to $\gamma$. But this argument actually only ensures that the traces of $\tilde{\gamma}$ and $\gamma$ are the same. To deduce that $\tilde{\gamma}$ is a CDE, we need to prove that it is a simple curve, or equivalently, that it cannot trace back its steps. This is implied by the following lemma.

\begin{lem}\label{lem:hairy}
The set of edges that are used by both the leftmost and the rightmost explorers is dense on the leftmost explorer trace in the scaling limit.
\end{lem}

\begin{proof}
Following the leftmost explorer, we apply crossing estimates in very thin rectangles (Theorem \ref{thm:cp}) at regularly spaced intervals to get paths of minusses (or {\em hairs}) that start at the very right of the exploration and reach to a positive distance (see Figure \ref{fig:hair}). The leftmost and the rightmost explorers being the same in the scaling limit, the rightmost explorer cannot go around the hairs, and hence has to meet the leftmost explorer at the base of every hair, which we just argued is a dense set in the scaling limit.
\end{proof}

\begin{figure}[htb]
\begin{center}
\includegraphics[width = 10cm]{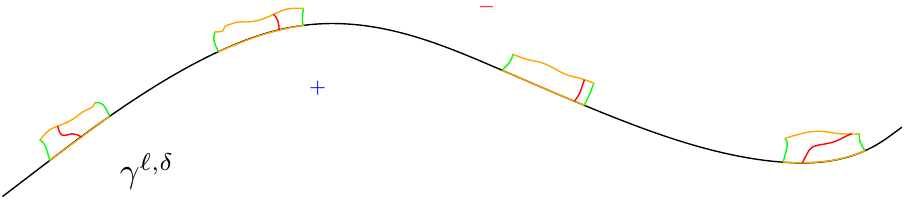}
\caption{Crossing estimates ensure with high probability that the explorer is hairy.}
\label{fig:hair}
\end{center}
\end{figure}

Note that an edge used by both the leftmost and rightmost explorers needs to be used by any exploration with the same endpoints, and moreover it constitutes a point of no-return for any such exploration. The fact that these edges form a dense set ensures that no exploration can trace back its steps on a macroscopic scale. This concludes the proof of the fact that arbitrary explorations between the same endpoints jointly converge towards the same CDE.

\section{Proof of Theorem~\ref{thm:main} (Convergence of crossing probabilities)}\label{sec:3}

Fix a topological rectangle $(\Omega,a,b,c,d)$. Let $a^\delta,b^\delta,c^\delta$ and $d^\delta$ on $\partial \Omega^\delta$ be discrete approximations of these points. Let $u$ be the boundary vertex of $\partial(\Omega^\delta)^*$ corresponding to the dual edge immediately clockwise to $a^\delta$, and let $v\in\partial(\Omega^\delta)^*$ be any boundary vertex adjacent to a vertex in $[b^\delta,c^\delta]$. Then

\begin{align*}\label{eq:expression}
\mu_{\Omega^\delta,\beta_c}^{\rm free} \big(\gamma^{\ell,\delta}_{u,v}\mbox{ hits }\left[c^\delta d^\delta\right]\mbox{ before }\left[b^\delta c^\delta\right]\big)&=\mu_{\Omega^{\delta},\beta_c}^{\mathrm{free}}\big( \left[a^\delta b^\delta\right]\overset{\star}{\leftrightsquigarrow}\left[c^\delta d^\delta\right]\big)\\
\mu_{\Omega^\delta,\beta_c}^{\rm free} \big(\gamma^{r,\delta}_{u,v}\mbox{ hits }\left[c^\delta d^\delta\right]\mbox{ before }\left[b^\delta c^\delta\right]\big)&=\mu_{\Omega^{\delta},\beta_c}^{\mathrm{free}}\big( \left[a^\delta b^\delta\right]{\leftrightsquigarrow}\left[c^\delta d^\delta\right]\big)
\end{align*}

Since $(\gamma^{\ell,\delta}_{u,v})$ and $(\gamma^{r,\delta}_{u,v})$ converge in law to the CDE from $a$ to $v$, and thanks to Lemma \ref{lem:boundary_touch}, we have :

$$
\lim_{\delta\to0}\mu_{\Omega^{\delta},\beta_c}^{\mathrm{free}}\big( \left[a^\delta b^\delta\right]\overset{\star}{\leftrightsquigarrow}\left[c^\delta d^\delta\right]\big)=\lim_{\delta\to0}\mu_{\Omega^{\delta},\beta_c}^{\mathrm{free}}\big( \left[a^\delta b^\delta\right]{\leftrightsquigarrow}\left[c^\delta d^\delta\right]\big)=f(\Omega,a,b,c,d)
$$

where 
\begin{equation*}f(\Omega,a,b,c,d):=\mathbb P\Big({\rm CDE}\mbox{ hits }\left[cd\right]\mbox{ before }\left[bc\right]\Big).\end{equation*}

Note that the exact position of the point $v\in[b,c]$ in the preceding argument does not matter. This is coherent with the fact that, according to Proposition \ref{prop:dipolar}, the law of {\rm CDE} until its first hitting time of $[bc]$ is independent of the choice of the observation point $v\in[bc]$.

\section{Proof of Theorem~\ref{thm:free-arc-cv} (Convergence of the arc ensemble)}\label{sec:4}

\subsection{Definition of the Free Arc Ensemble}
We somehow follow Schramm \cite{S0} to describe the Free Arc Ensemble (FAE). 

For a metric space $X$, let $\mathcal{H}(X)$ be the set of compact subsets of $X$ equipped with the Hausdorff topology, i.e. the topology induced by the distance
$$
d_{\mathcal{H}(X)}(K,K')=\max_{x\in K} \min_{x'\in K'}\{d_X(x,x'),1\}
$$

Recall that for a Jordan domain $\Omega$ of the plane, we denote by $\mathcal{C}\left(\overline{\Omega}\right)$ the set of continuous maps from $[0,1]$ to $\overline{\Omega}$ equipped with the topology of uniform convergence up to reparametrization.

The FAE is a random element of $\mathcal{H}\left(\partial\Omega\times\partial\Omega\times\mathcal{C}\left(\overline{\Omega}\right)\right)$. Its measure is supported in a specific subset of it: the FAE can almost surely be written as
$$
\bigcup_{u,v\in\partial\Omega} (u,v,\gamma_{u,v})
$$
where $\gamma_{u,v}$ is generically a simple path from $u$ to $v$ in $\overline{\Omega}$ -- or the degenerate path $\{u\}$ when $u=v$. However, for (countably many) exceptional couples of boundary points $(u,v)$, $\gamma_{u,v}$ is a set that consists of two simple paths from $u$ to $v$ in $\overline{\Omega}$ -- or even of two simple loops rooted at $u$ and of the degenerate path $\{u\}$ when the boundary point $u=v$ is exceptional. Therefore, $\gamma_{u,v}$ can also be the union of two or three curves instead of one.

\medbreak
Let us construct the FAE as follows:

\begin{enumerate}
\item First, fix $u\in\partial\Omega$ and let $(v_1,\dots,v_n,\dots)$ be a dense subset of $\partial\Omega$. We construct the triplets $(u,v_i,\gamma_{u,v_i})$ recursively. Let $\gamma_{u,v_1}$ be a CDE from $u$ to $v_1$. The exploration $\gamma_{u,v_i}$ starts by following the previously built exploration paths $(\gamma_{u,v_j},j<i)$ by choosing, at each branching point, to go in the region that contains $v_i$ on its boundary (see Figure \ref{fig:branch}) At the first branching point for which there is no such possible choice, the path $\gamma_{u,v_i}$ continues as a CDE from $u$ to $v_i$ would (this is possible thanks to Proposition \ref{prop:dipolar}).

\begin{figure}[htb]
\begin{center}
\includegraphics[width = 10cm]{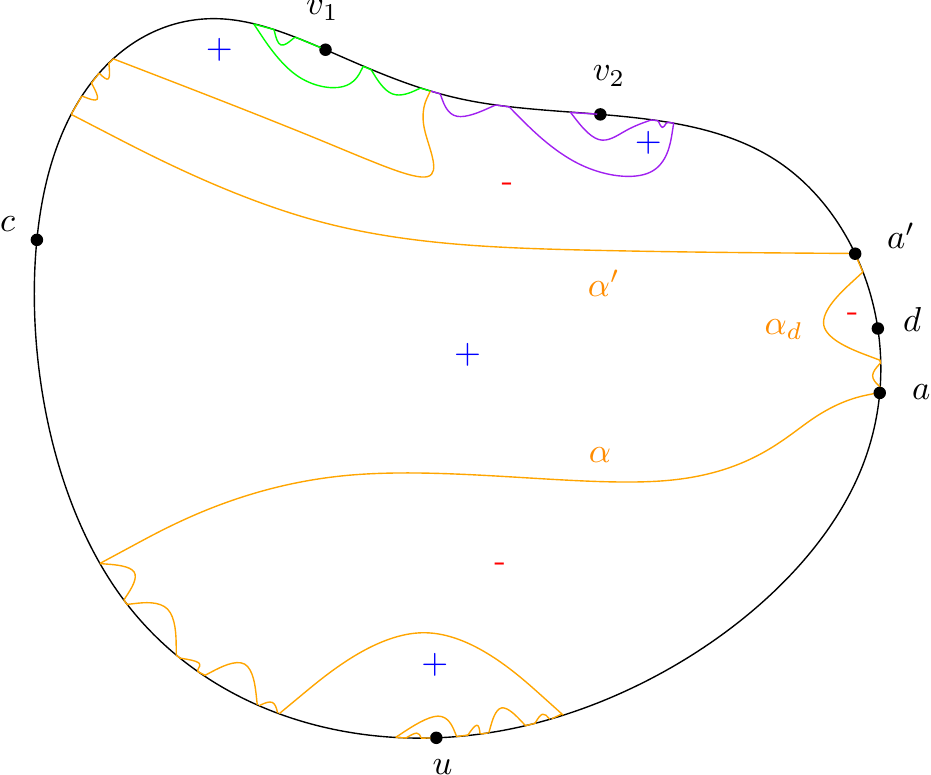}
\caption{The coupling of $\gamma_{u,v_1}$ (orange then green) and $\gamma_{u,v_2}$ (orange then purple). The $+$ and $-$ show what the spins are on the boundary of the subdomains.}
\label{fig:branch}
\end{center}
\end{figure}

\noindent{\em Remark.} The union of the traces of the $\gamma_{u,v_i}$ for $i\ge1$ can be decomposed into countably many arcs (i.e. excursions out of the boundary), which come with a natural orientation depending on the orientation of the $\gamma_{u,v_i}$. The ends of these arcs form a countable (dense) subset $E$ of the boundary $\partial \Omega$.

\item Let us now build $\gamma_{u,v}$ for an arbitrary boundary point $v$ on $\partial \Omega$: simply take the limit of $\gamma_{u,v_i}$ as $v_i \rightarrow v$, or alternatively  follow the tree composed of the $(\gamma_{u,v_i},i\ge1)$ heading systematically towards $v$.

\item Let us describe how to build $\gamma_{u',v'}$ for $u',v'\in\partial\Omega\setminus E$. 
Consider the set of arcs disconnecting $u'$ from $v'$. It comes with a natural ordering (that makes it isomorphic to $\mathbb{Z}$). We call these arcs {\em sideways arcs}. The path $\gamma_{u',v'}$ will go through all of the sideways arcs in order, and we now need to explain how it will go from one sideways arc to the next. 

Let $\alpha$ and $\alpha'$ be two successive sideways arcs. Suppose the arc $\alpha$ crosses from left to right seen from $u'$, and call $[aa']$ the right piece of boundary between $\alpha$ and $\alpha'$ (where $a$ is on $\alpha$ and $a'$ on $\alpha'$). Choose also a boundary point $c$ on the left piece of boundary between $\alpha$ and $\alpha'$. For any point $d \in [aa']$, let us call $\alpha_d$ the largest arc disconnecting $c$ and $d$ that still has its two endpoints on $[aa']$. Between $\alpha$ and $\alpha'$, the path $\gamma_{u',v'}$ follows the arcs $\alpha_d$ in their order of appearance i.e. as $d$ follows $[aa']$ from $a$ to $a'$ (see Figure \ref{fig:branch}). One can easily check that $\gamma_{u',v'}$ can follow these arcs with their natural orientation.

\item Let us now build $\gamma_{u',v'}$ whenever $u'$ and $v'$ are not the two ends of a common arc. Assume for instance that $u'\in\partial\Omega\setminus E$ and $v'$ is the end of an arc $\alpha$ (other situations can be handled similarly). Choose a point $w\in\partial\Omega\setminus E$ that is disconnected from $u'$ by the arc $\alpha$. By construction, $\gamma_{u',w}$ goes through $v'$, and so we let $\gamma_{u',v'}$ be the subpath of $\gamma_{u',w}$ going from $u'$ up to $v'$.

\item We now explain how to build $\gamma_{u',v'}$ when $u',v'$ are the two ends of some arc $\alpha$. If $\alpha$ goes from $u'$ to $v'$, then simply let $\gamma_{u',v'}=\alpha$. If the arc $\alpha$ is oriented from $v'$ to $u'$, then $\gamma_{u',v'}$ is a set of two simple paths, one going to the left of $\alpha$, the other going to its right. The construction of these two simple paths is in the same spirit as above. For example, the path going to the right of $\alpha$ can be build in the following way. Let us call $[u',v']$ the part of the boundary $\partial\Omega$ sitting to the right of the arc $\alpha$. Choose also a boundary point $c$ on the boundary arc to the left of $\alpha$. For all points $d \in[u',v']$, let us call $\alpha_d$ the largest arc disconnecting $c$ and $d$ that stays strictly to the right of $\alpha$. The right exploration path from $u'$ to $v'$ will follow the arcs $\alpha_d$ in their order of appearance (i.e. as $d$ follows $[u'v']$ from $u'$ to $v'$). 

\item To finish our construction, we need to describe the (set of) path(s) $\gamma_{u',u'}$. One choice is the constant curve $\{u'\}$. If moreover an arc $\alpha$ starts or ends at $u'$, $\gamma_{u',u'}$ will moreover contain two simple loops that can be build by concatenation of $\gamma_{u',w}$ and $\gamma_{w,u'}$, where $w$ is the other end of the arc $\alpha$ (one of these sets is a singleton and the other a pair, which one is which depending on the orientation of the arc $\alpha$).
\end{enumerate}

\begin{defn}
The {\em Free Arc Ensemble} is the union FAE of triples $(u,v,\gamma_{u,v})$ coupled as described above, where $u$ and $v$ run over all couples of boundary points and $\gamma_{u,v}$ is the set of possible explorations from $u$ to $v$ (there can be one, two or three such explorations).
\end{defn}

It will be technically easier (for instance to get tightness) for us to use an alternate way to describe FAE, that we call FAE' until proven to be the same as FAE. The construction differs from the one of FAE only in its first step. We choose a countable dense set of points $\mathcal{P}= \{u_1,u_2,\ldots\}$ on $\partial\Omega$. We couple all the paths $\gamma_{u_i,u_j}$ by induction.
\begin{itemize}
\item[(1a)] Take $\gamma_{u_1,u_2}$ to be a CDE, and let $\gamma_{u_2,u_1}$ be given conditionally on $\gamma_{u_1,u_2}$ as follows (see Figure \ref{fig:b_a_f}): it takes the same sideways arcs. It goes from one arc to the next as conditionally independent SLE$(3,\frac{-3}2)$ in each domain\footnote{The scaling limits of the discrete interfaces satisfy this property : it is a statement very similar to Theorems \ref{thm:ci} and \ref{thm:dde-to-cde-cv} and the proof would follow the same lines.}, where the forced point starts at $0^+$ or $0^-$ depending on boundary conditions which we recover from the orientation of the path $\gamma_{u_1,u_2}$.

\begin{figure}[htb]
\begin{center}
\includegraphics[width = 10cm]{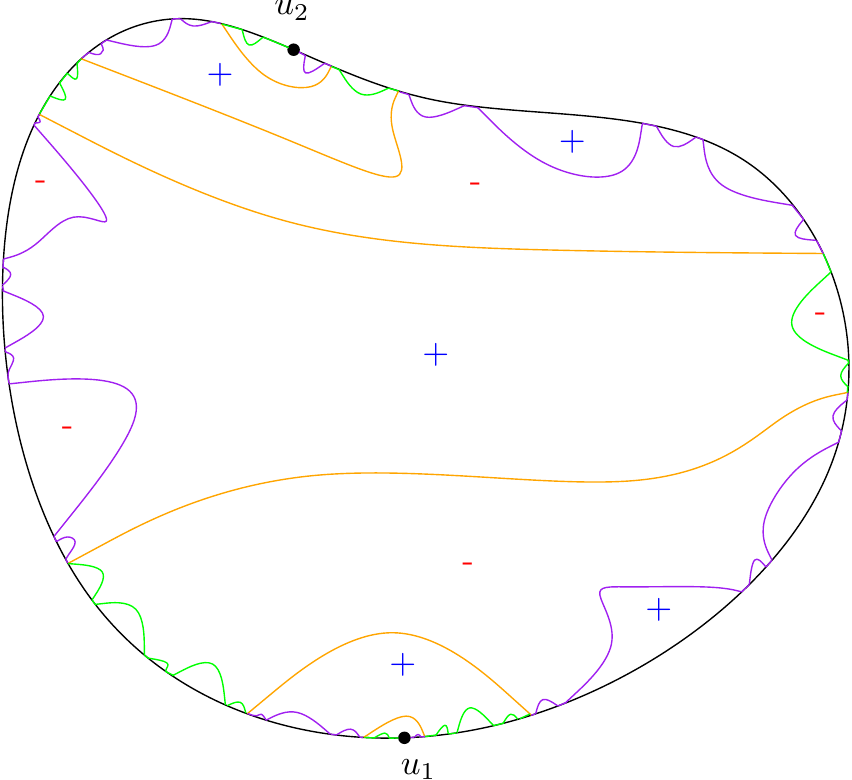}
\caption{The coupling of $\gamma_{u_1,u_2}$ (orange and green) and $\gamma_{u_2,u_1}$ (orange and purple). The purple curves represent SLE$(3,\frac{-3}2)$ in each of the subdomains cut up by the orange and green curves.}
\label{fig:b_a_f}
\end{center}
\end{figure}

\item[(1b)] Now, assume that $\gamma_{u_i,u_j}$ are constructed for every $i,j<n$. In the family of arcs already constructed, there is a unique one, denoted by $\alpha$, that separates $u_n$ from all the other points $u_i$. The path $\gamma_{u_i,u_n}$ follows the exploration tree starting from $u_i$, until it takes the arc $\alpha$. After $\alpha$, all paths $\gamma_{u_i,u_n}$ follow the same SLE$(3,\frac{-3}2)$ in the connected component of $\Omega\setminus \alpha$ containing $u_n$. The paths $\gamma_{u_n,u_i}$ all follow the same trajectory until the arc $\alpha$, which can be build as in (1a). Once the arc $\alpha$ has been crossed, there is a unique way to follow arcs that have already been discovered in order to reach $u_i$, and this is exactly how $\gamma_{u_n,u_i}$ gets to its goal.
\end{itemize}

\begin{prop}\label{prop:FAEclose}
The arc ensemble {\rm FAE'} is the closure of the set of paths with endpoints in $\mathcal{P}$.
\end{prop}

\begin{proof}
Both the fact that FAE' is included in the closure of the set of paths with endpoints in $\mathcal{P}$ and the fact that FAE' is closed can be checked from the definition of the arc ensemble FAE', by considering the different cases.
\end{proof}

\subsection{Proof of Theorem~\ref{thm:free-arc-cv} (Convergence of the arc ensemble)}

The proof of Theorem~\ref{thm:free-arc-cv}  follows from:

\begin{thm}
\label{thm:free-arc-cv-ii}Let $\Omega$ be a simply connected Jordan domain. Consider the critical Ising model on $\Omega^\delta$ with free boundary conditions.
The family $(\mathcal{A}^{\delta})_{\delta>0}$ converges in law (when the mesh size $\delta$ goes to $0$) to the arc ensembles {\rm FAE} and {\rm FAE'}.
\end{thm}

\begin{cor}
The arc ensembles {\rm FAE} and {\rm FAE'} have same law, which is moreover independent of the choice of points made during their constructions.
\end{cor}

The first step of the proof is a tightness lemma:

\begin{lem}
The family $(\mathcal{A}^\delta)_{\delta>0}$ is tight.
\end{lem}

\begin{proof}Consider $\varepsilon>0$ and let $\eta=\eta(\varepsilon)>0$ to be chosen shortly. Choose a finite subset $\{u_1,\dots,u_N\}$ of $\mathcal{P}$ that cuts the boundary of the domain $\Omega$ in arcs of diameter less than $\eta$. The set of all leftmost and rightmost explorations starting from and aiming at the points $u_1,\dots,u_N$ form a tight family for the topology of uniform convergence up to reparametrization (as each explorer is close to a CDE). By chosing $N$ big enough, we can ensure that no boundary point can be connected in $\Omega$ to points $\epsilon$-inside the domain without crossing one of the finitely many excursions (thanks to the behavior of CDE, and by Lemma \ref{lem:boundary_touch}).

Using Lemma \ref{lem:hairy}, we see that with high probability, no exploration starting and ending at points of $\partial\Omega^\delta$ goes out of the $\varepsilon$-neighborhood of the explorations starting and ending at points of $\{u_1,\dots,u_N\}$. The tightness then follows from the tightness of single exploration paths.\end{proof}

\begin{proof}[Proof of Theorem~\ref{thm:free-arc-cv-ii}]Consider a sub-sequential limit $\mathcal{A}$ of $(\mathcal A^\delta)_{\delta>0}$. From the convergence of explorations Theorems \ref{thm:dde-to-cde-cv} and \ref{thm:ci}, we see that the explorations with endpoints in $\mathcal{P}$ have the same joint law as the corresponding explorations in FAE'. By definition, $\mathcal{A}$ is closed, and hence ${\rm FAE}'\subset\mathcal{A}$ by Proposition \ref{prop:FAEclose}. To prove the reverse inclusion $\mathcal{A}\subset{\rm FAE}'$, it is enough to show that
almost surely, any path in $\mathcal{A}$ is in the closure of the set of paths with endpoints in $\mathcal{P}$.
But this follows directly from the argument given for tightness.

We can now show that FAE and FAE' are the same object. Indeed, consider the limit of the discrete exploration tree from a point $u$ to a countable dense set $v_i$. This is a subset of $\mathcal{A}$, which has the same joint law as the exploration tree of FAE. Moreover this tree contains all the arcs appearing in the set of paths of $\mathcal{A}$ with endpoints in $\mathcal{P}=\{u,v_1,v_2,\ldots\}$ and so the whole limit $\mathcal{A}={\rm FAE}'$ can be rebuilt from the arcs of this exploration tree alone. In particular, the arc ensembles FAE and FAE' have the same law.
\end{proof}

\bigbreak\noindent
\textbf{Acknowledgments.} The authors thank Julien Dub\'{e}dat, Stanislav Smirnov and Wendelin Werner for interesting and stimulating discussions. The research of H. D.C. was supported by the ERC grant CONPASP and the NCCR Swissmap founded by the Swiss NSF. The research of C. H. was supported by the NSF grant DMS-1106588 and the Minerva Foundation.

\bibliography{biblio}{}
\bibliographystyle{plain}

\end{document}